\documentclass[12pt]{amsart}

\usepackage{amssymb}
\usepackage{enumerate}
\usepackage{amsmath}
\usepackage{amsmath}
\usepackage{amsfonts}
\usepackage{amsthm}
\usepackage{fullpage}
\usepackage{epsfig}

\newtheorem{theorem}{Theorem}
\newtheorem{lemma}[theorem]{Lemma}
\newtheorem{remark}[theorem]{Remark}
\newtheorem{corollary}[theorem]{Corollary}

\newtheorem{proposition}[theorem]{Proposition}
\newtheorem{definition}[theorem]{Definition}

\def\Re{\mathop{\rm Re}\nolimits}

\def\erf{\mathop{\rm erf}\nolimits}
\def\Index{\mathop{\rm Index}\nolimits}
\def\Im{\mathop{\rm Im}\nolimits}

\newcommand{\C}{\mathbb C}
\newcommand{\R}{\mathbb R}

      \makeatletter
      \def\@setcopyright{}
      \def\serieslogo@{}
      \makeatother

\begin{document}

   \author{Dan Kucerovsky}
   \address{D. Kucerovsky, Department of Mathematics, University of New Brunswick, Fredericton, NB E3B 5A3, Canada}
   \email{dan@math.unb.ca}

  \author{Amir T. P. Najafabadi}

   \address{A. T. P. Najafabadi, Department of Mathematical Sciences, Shahid Beheshti University, G.C. Evin, 1983963113, Tehran, Iran}

\email{amirtpayandeh@sbu.ac.ir}

   \author{Aydin Sarraf}

   \address{A. Sarraf, Department of Computer Science, University of New Brunswick, Fredericton, NB E3B 5A3, Canada}

\email{Aydin.Sarraf@unb.ca}

   \title[On the Riemann-Hilbert factorization problem for positive definite functions]{On the Riemann-Hilbert factorization problem for positive definite functions}

   \begin{abstract}
    We give several general theorems concerning positive definite solutions of  Riemann-Hilbert problems on the real line. Furthermore, as an example, we apply our theory to the characteristic function of a class of L\'{e}vy processes and we find the distribution of their extrema at a given stopping time.
   \end{abstract}

 \maketitle


 \section{Introduction}
 The Wiener-Hopf factorization method was initially used for solving certain integral equations \cite{hopf1932klasse}. This method leads to Wiener-Hopf equations which can also be recast, by taking a Fourier transform, in the form of a Riemann-Hilbert problem \cite{hilbert,riemann}. Although there is a well-known uniqueness theory based on the index for Riemann-Hilbert problems, the function class usually considered is H\"{o}lder continuous functions going to one at infinity, and as will be seen, this is not always the right class for applications. We consider the problem of finding positive definite solutions to Riemann-Hilbert problems on the real line.
This problem has a very natural connection with the theory of characteristic functions.
This is because, by Bochner's theorem, a positive definite function $\phi:\mathbb{R}^n\longrightarrow \mathbb{C}$ which is continuous at the origin and satisfies the equation $\phi(0)=1$, is the characteristic function of some random variable and vice versa.
Therefore, positive definite functions and characteristic functions are nearly synonymous. Furthermore, characteristic functions and probability density functions are intimately connected in the sense that probability density functions are inverse Fourier transforms of characteristic functions and characteristic functions are Fourier transforms of probability density functions. Applications of Wiener-Hopf factorization methods in finding the distribution of certain random variables in statistics dates back to Baxter, Donsker \cite{baxter1957distribution} and Pecherskii, Rogozin \cite{pecherskii1969joint,rogozin1966distributions}. Since characteristic functions are positive definite, we are interested in positive definite solutions of Riemann-Hilbert problems. This is particularly important when the solution cannot be shown to be unique.

In the usual uniqueness theory for Riemann-Hilbert factorizations, the key condition that must be imposed on the function $f$ is that it must go to a nonzero constant at infinity: this condition is needed to insure that the so-called index of the function $f$ is well-defined and stable under small perturbations (see Definition \ref{def:index} and the discussion there). As will be seen later, the theory of L\'evy processes gives natural examples of Riemann-Hilbert problems where we seek to factorize a characteristic function as the product of two characteristic functions, but the condition of going to a nonzero constant at infinity is too strong. Thus the usual uniqueness theory for Riemann-Hilbert problems becomes problematic. However, what is really needed in this application is not so much the uniqueness theory as just the existence of a factorization where both factors are positive definite, as in the abovementioned result.

 The classical multiplicative homogeneous Riemann-Hilbert problem is to determine a sectionally analytic function $\Phi$ with a jump discontinuity on an oriented contour $L$ consisting of a finite number of simple smooth closed or open curves, and with sections $\Phi_+$ and $\Phi_-$ constituting the boundary values on $L$, which satisfies on $L$ the boundary condition $G(t)=\Phi_+(t)\Phi_-(t)$ where the function $G$, $G(t)\neq 0$; satisfies the H\"{o}lder condition everywhere except at a finite number of points where it may have discontinuities of the first kind. 

We restrict ourselves to the case where the contour is the real line, because the Bochner theory of positive definite functions is only available for a very few contours, the cases of the real line and the unit circle are the only ones that appear to be tractable.

Our main result (Corollary \ref{cor:factorization}) states that under certain conditions, a function $f$ on the real line can be factorized in the form $f=\Phi_{+} \Phi_{-}$ with $\Phi_{+}$ and $\Phi_{-}$ are positive definite if and only if $f$ is positive definite.

 A significant difficulty with the classical singular integral solutions to Wiener-Hopf equations in particular and Riemann-Hilbert problems in general, is that these solutions often involve integrals that are both analytically intractable and difficult to evaluate numerically (c.f. \cite{kucerovsky2009approximation}). In this paper, we develop a method which gives positive definite solutions to certain Riemann-Hilbert problems. Furthermore, our method doesn't suffer from the aforementioned difficulties and gives exact results that can be truncated with a known error term to provide accurate numerical approximations directly.\\ 
\section{Definitions and Preliminaries}
 We define positive definiteness for functions as follows:

\begin{definition} A function on the real line is  positive definite if and only if the Fourier transform exists everywhere and takes values in the non-negative real numbers.
\end{definition}

We define positive semi-definiteness for matrices as follows:

\begin{definition} A matrix is \emph{positive semi-definite} if and only if it is Hermitian and has non-negative eigenvalues.\end{definition}

It was shown by Bochner \cite{bochner1933monotone} that the following Theorem holds:
\begin{theorem} A function $f$ is positive definite if and only if the matrices of the form
$[f(x_i-x_j)]$  are positive semi-definite matrices, for all choices of finite subsets of points $\{x_i\}\subset \R.$ \label{th:bochner}
\end{theorem}

Schur has obtained the following interesting property of positive semi-definite matrices \cite{schur1911bemerkungen}:
\begin{theorem}\label{schur} The elementwise product $[a_{ij}b_{ij}]$ of positive semi-definite matrices $[a_{ij}]$ and $[b_{ij}]$ is positive semi-definite.\label{th:schur}\end{theorem}

Combining these two Theorems, we prove the following:
\begin{theorem} If $f(x)$ is positive definite then $\exp(k f)$ is positive definite (for all scalars $k>0$). If $f(x)^\lambda$ is positive definite for all $\lambda>0,$ and $f$ is bounded away from zero, then $\log f$ differs by a constant from a positive definite function.
\label{th:positivedef}\end{theorem}
\begin{proof} First of all, we show that if $f$ is positive definite, then so is   $\exp(k f).$
By Theorem \ref{th:bochner}, this is equivalent to showing that if matrices of the form
$[f(x_i-x_j)]$ are positive semi-definite, then so are the matrices $[\exp(kf(x_i-x_j))].$ But
by Theorem \ref{th:schur} we have first of all that if $[f(x_i-x_j)]$ is positive semi-definite
then the same holds for $[f(x_i-x_j)^m],$ where  $m$ is  a positive integer. The exponential function has an entire Taylor series expansion $\exp( k t)=\sum_0^\infty \frac{k^m t^m}{m!}$ where the coefficients are positive, so  $[\exp(kf(x_i-x_j))]$ can be written as a convergent sum of positive semi-definite matrices, and is hence positive semi-definite. The second claim is proven in \cite[p.531]{schoenberg1938metric}, see also \cite{von1941fourier}.
 \end{proof}

Finally, we recall the following:
\begin{theorem}
\label{Shannon} Suppose that $f$ has a Fourier transform that is zero outside the interval  $[-\frac{\pi}{h},\frac{\pi}{h}].$ Then,
\begin{equation*}
f(x)=\sum_{-\infty}^{\infty} f\left(nh\right)\frac{\sin\frac{\pi}{h}(x-nh)}{\frac{\pi}{h}(x-nh)}.
\end{equation*}
More generally, if $f$ is continuous, square-integrable, and is represented as the Fourier transform of a function $\hat{f}$ in $L^1(\R)$ then
$$\left|f(x)-\sum_{-\infty}^{\infty} f\left(nh\right)\frac{\sin\frac{\pi}{h}(x-nh)}{\frac{\pi}{h}(x-nh)}\right|$$ is bounded by $$ \frac{1}{\pi}\int_{|\omega|\geq \frac{\pi}{h}} | \hat{f}(\omega)|\,d\omega.$$
\end{theorem}

The above theorem has a complicated and interesting history, which is reviewed in \cite{splettstosser198375}. There exist a multitude of early versions of the theorem, such as in \cite{splettstosser1978some}, or the very early paper by de la Vall\'ee Poussin \cite{poussin1908convergence}. The first part of the theorem as quoted above comes from \cite{shannon1949communication}, see also \cite{whittaker1915xviii}, and for the second part see  \cite[Theorem B]{butzer2005classical}.

The above Theorem is often used to expand functions whose Fourier transform decays at infinity, rather than actually being zero outside a bounded interval.  Thus, for example, Rybicki \cite{rybicki1989dawson}
shows by applying the above Theorem that
$$ F(z) = \lim_{h\longrightarrow0+} \frac{1}{\sqrt{\pi}} \sum_{\mbox{$n$ \small odd}} \frac{\exp\{-(z-nh)^2\}}{n},$$
where $F(z)$ happens to be the complex Dawson's integral. We now give a Corollary that generalizes Rybicki's observation:\\
\begin{corollary} Suppose that $f\in L^2(\R)$ is positive definite and continuous.  Then, $f$ is represented as the Fourier transform of some function in $L^1(\R ),$ and
$$ f(x) = \lim_{h\longrightarrow 0+} \sum_{n=-\infty}^{\infty}f\left(nh\right)\frac{\sin(\pi(x/h-n))}{\pi(x/h-n)}$$
where the limit is with respect to the supremum norm.\label{cor:sh}
\end{corollary} \begin{proof}
If $f$ is positive definite, then by definition its Fourier transform $\hat{f}$ is real-valued and non-negative everywhere. By the Fourier inversion formula, which holds everywhere since $f$ is continuous, we then have that $f(0)=\frac{1}{2\pi i} \int_{-\infty}^{\infty} {\hat{f}(\omega)}\,d\omega.$ Since $\hat{f}$ is non-negative, this integral equals the $L_1$-norm of $\hat{f},$ and is finite because the continuity of $f$ insures that $f(0)$ is finite.

Let us consider the Shannon--Whittaker series expansion for some arbitrary $h>0$ :
$$S_{h}(x):=\sum_{-\infty}^{\infty} f\left(nh\right)\frac{\sin\frac{\pi}{h}(x-nh)}{\frac{\pi}{h}(x-nh)} .$$
We see that if $f$ is continuous, square-integrable, and is represented as the Fourier transform of a function $\hat{f}$ in $L^1(\R)$ then
$$|S_h(x)-f(x)|\leq \frac{1}{\pi}\int_{|\omega|\geq \frac{\pi}{h}} | \hat{f}(\omega)|\,d\omega.$$
Taking the limit as $h\longrightarrow0+,$ we obtain uniform convergence.
\end{proof}

We recall the classic form of the Paley-Wiener theorem:

\begin{theorem}
\label{pw}
Suppose $f\in L^2(\mathbb{R})$.  The following are equivalent:

\begin{enumerate}\item[ (i)] The real function $f$ vanishes on $\mathbb{
R}^{-}$;
 \item[ (ii)] The Fourier transform $\hat {f}$ of $f$ extends to a holomorphic function on the upper half-plane and
the $L_2$-norms of the functions $x\mapsto\hat {f}(x+iy_0)$ are
continuous and uniformly bounded for all $y_0\geq 0.$
\end{enumerate}
\end{theorem}
\section{Riemann-Hilbert problem for positive definite functions}
Let $L$ be an oriented contour which consists of a finite number of simple smooth closed or open curves situated in an arbitrary way on the plane. The classical additive (multiplicative) homogeneous Riemann-Hilbert problem is to determine a sectionally analytic function $\Phi$ with a jump discontinuity on $L$ and with sections $\Phi_+$ and $\Phi_-$ constituting the boundary values on $L$, which satisfies on $L$ the boundary condition $\Phi_+(t)-\Phi_-(t)=G(t)$ ( $\Phi_+(t)=G(t)\Phi_-(t)$ ) where $G$ is a function satisfying the H\"{o}lder condition everywhere except at a finite number of points where they may have discontinuities of the first kind, and $G(t)\neq 0$ (See \cite[section 44.1]{gakhovboundary}). In the following Theorem, the function $f$ does not need to satisfy the  H\"{o}lder condition.

\begin{theorem}
Suppose that $f$ is in $L^2(\R),$ is positive definite, and continuous. The additive homogeneous Riemann-Hilbert problem for $f$ is solved by
a decomposition  $f_{-}+f_{+}$ where $f_+$ has Fourier transform supported in $[0,\infty)$ and $f_-$ has Fourier transform supported in $(-\infty,0].$ The functions in the decomposition are positive definite.  We furthermore have
 $$f_{+}(w)=\lim_{h\longrightarrow0}\sum_{n=-\infty}^{\infty}f\left(nh\right)\frac{\exp(\pi i(w/h-n))-1}{2\pi i(w/h-n)}$$ and $$f_{-}(w)=\lim_{h\longrightarrow0}\sum_{n=-\infty}^{\infty}f\left(nh\right)\frac{1-\exp(-i\pi(w/h-n))}{2\pi i(w/h-n)}.$$ \label{th:additive.RH}
\end{theorem}
\begin{proof}  We can decompose uniquely so that $f=f_{-}+f_{+}$ where the Fourier transform of $f_{+}$ is the restriction of the Fourier transform of $f$ to $\R^+ ,$ and similarly for $f_-$. If $f$ is positive definite, then clearly so are $f_+$ and $f_-.$ Theorem \ref{pw} shows that the holomorphicity conditions needed for the additive Riemann-Hilbert problem are satisfied by $f_+$ and $f_-$.  We observe that $f_{\pm}$ is in fact continuous: this is because the first part of Corollary \ref{cor:sh} shows that the Fourier transform $\hat{f}$ of $f$ is in $L^1(\R).$ The same is then true for the Fourier transforms of $f_{+}$ and $f_{-}$, and the Riemann--Lebesgue lemma then implies that $f_{+}$ and $f_{-}$ are continuous. We now consider the boundedness of $f_{+}$ in the upper half plane. Since $f_{+}$ is continuous and positive definite, the first part of Corollary \ref{cor:sh} shown that $f_{+}(x)=\int_0^\infty \hat{f}(\omega) e^{i\omega x}\,d\omega, $ which converges for all $x$ in the upper half plane. But since $| e^{i\omega x} | \leq1$ for all real $\omega$ and all $x$ in the upper half plane, we see that $|f_{+}(x)|\leq \int_0^\infty |\hat{f}(\omega)|\,d\omega.$
 The series decompositions come from a slight generalization of the second part of Corollary \ref{cor:sh} to the case of functions supported on $[0,\pi/h]$ or $[-\pi/h,0].$
\end{proof}

\begin{proposition}
Suppose that $f:\mathbb{R}\longrightarrow \mathbb{C}$ is positive definite, has Fourier transform supported on $[-\sigma,\sigma]$ where $0<\sigma<\pi$ and
\begin{equation*}
\lim_{n\longrightarrow\pm\infty}\frac{\Re(f(n))}{n^2}=\lim_{n\longrightarrow\pm\infty}\frac{\Im(f(n))}{n}=0
\end{equation*}

The function $f$ has a decomposition $f_{-}+f_{+}$ where $f_+$ has Fourier transform supported in $[0,\sigma]$ and $f_-$ has Fourier transform supported in $[-\sigma,0].$ The functions in the decomposition are positive definite.  We furthermore have
 $$f_{+}(w)=\sum_{n=-\infty}^{\infty}f\left(n\right)\frac{\exp(\pi i(w-n))-1}{2\pi i(w-n)}$$ and $$f_{-}(w)=\sum_{n=-\infty}^{\infty}f\left(n\right)\frac{1-\exp(-i\pi(w-n))}{2\pi i(w-n)}.$$

 where the convergence is uniform on compact subsets of $\mathbb{R}$.
\end{proposition}

\begin{proof}
The statements on uniform convergence follow from \cite[Theorem 1]{bailey2015cardinal}. Since $f$ is positive definite, $f(-x)=\overline{f(x)}$ implying that $f_e(x)=\frac{f(x)+f(-x)}{2}=\Re(f(x))$ and $f_o(x)=\frac{f(x)-f(-x)}{2}=i\Im(f(x))$.
\end{proof}

\begin{theorem}\label{main}  Suppose that $f$ is a function on the real line having the properties:
\begin{enumerate}[i)]
	\item $f\in L^2(\R)+\C 1,$ \item $|f|$ is bounded and bounded away from zero, \item and $f^\lambda$ is positive definite for all $\lambda>0$.
\end{enumerate} Then $f=\Phi_{+}\Phi_-$ where $\Phi_+$ is positive definite on $\R$ and extends to a bounded holomorphic function on the upper half plane, and $\Phi_-$ is positive definite on $\R$ and extends to a bounded holomorphic function on the lower half plane.
\end{theorem}
\begin{proof} By the second and third hypotheses, $\log(f)$ differs by a constant from a positive definite function (Theorem \ref{th:positivedef}). It is straightforward to show that the first two hypotheses imply $\log(f)$ is in $L^2(\R)+\C 1.$ The solution to the additive Riemann-Hilbert problem (Theorem \ref{th:additive.RH}) then implies that $$\log(f)=g_{+}+g_{-}+\lambda,$$ where $\lambda$ is a constant, $g_+$ is holomorphic and bounded in the upper half plane, and $g_-$ is holomorphic and bounded in the lower half plane. Exponentiating, we notice that since the exponential function is entire, $\exp(g_{+}+\frac{\lambda}{2})$ is therefore holomorphic on the upper half plane and $\exp(g_{-}+\frac{\lambda}{2})$ is holomorphic on the lower half plane. Boundedness in these domains follows from the boundedness of $g_{\pm}.$ Furthermore, since $g_\pm$ was positive definite on $\R,$  the first part of Theorem \ref{th:positivedef} implies that $\Phi_{\pm}:=\exp(g_{\pm}+\frac{\lambda}{2})$ is positive definite on $\R.$ Moreover, $\Phi_+$ and $\Phi_-$ have no zeros in the upper and lower half planes respectively because $g_+$ and $g_-$ have no poles in the upper and lower half planes respectively.
\end{proof}

\begin{remark}
The general inhomogeneous problem can then be solved by combining the homogeneous multiplicative problem with the homogeneous additive problem.
\end{remark}

\begin{corollary}\label{cor:factorization}
If function $f$ satisfies the assumptions of Theorem \eqref{main}, then $f$ is positive definite if and only if there exist positive definite functions $\Phi_+$ and $\Phi_-$ such that $f=\Phi_+\Phi_-$.
\end{corollary}
\begin{proof}
By Theorem \eqref{main}, if $f$ is positive definite then there exist positive definite functions $\Phi_+$ and $\Phi_-$ such that $f=\Phi_+\Phi_-$. Its converse holds because product of positive definite functions is positive definite.
\end{proof}
In the proof of Theorem \eqref{main}, the expression $\log(f)$ becomes singular if $f$ vanishes at infinity on the real line. Nevertheless, we can handle this case by perturbing the function $f$ to $\tilde{f}=f+\epsilon$ for some $\epsilon>0$. It is easy to see that the perturbation $\tilde{f}$ also satisfies the assumptions $(ii)$ and $(iii)$ of Theorem \eqref{main}. The following Lemma shows that the process of exponentiating and anti-exponentiating which are used in constructing the solution are well-behaved with respect to approximation. 

\begin{lemma}
 If $x$ and $y$ are two points in $G=\{z\in \mathbb{C}: N<|z|<M\}$ such that the segment joining them is also in $G$ then  

\begin{equation*}
\exp(N)|x-y|<|\exp(x)-\exp(y)|<\exp(M)|x-y|
\end{equation*}




\end{lemma}

\begin{proof}
By \cite[Theorem 10]{mcleod1965mean}, $\exp(x)-\exp(y)=(x-y)(\lambda_1\exp(w_1)+\lambda_2\exp(w_2))$ where $\lambda_1$, $\lambda_2\geq 0$, and $\lambda_1+\lambda_2=1$ and $w_1$, $w_2$ are on the segment joining $x$ and $y$. Therefore, we conclude

\begin{equation*}
\exp(N)|x-y|<|\exp(x)-\exp(y)|<\exp(M)|x-y|
\end{equation*}


\end{proof}
 \section{An example: Distribution of Extrema of a Class of L\'{e}vy Processes}

We thus have a theory that allows us to construct positive definite solutions to Riemann-Hilbert problems in a very simple and explicit form, well-suited to numerical computation when needed. Since positive definite functions are, up to a normalization, the same as the characteristic functions that appear in statistics, we seek statistical applications. We note that in the case of characteristic functions, the Riemann-Lebesgue lemma implies that they are in $C_0(\mathbb{R})$, and thus the usual uniqueness theory based on the index for Riemann-Hilbert problems is not directly applicable. In fact, the index of a function in $C_0(\mathbb{R})$ is undefined (see \cite{gakhovboundary}, however, for some special cases). The usual definition of the index is as follows:

 \begin{definition}\cite[Definition 12.1]{gakhovboundary}
Let $L$ be a smooth closed contour and $G(t)$ a continuous non-vanishing function given on $L$. By the index of the function $G(t)$ with respect to the contour $L$ we understand the increment of its argument, in traversing the curve in the positive direction, divided by $2\pi$.\label{def:index}
\end{definition}


 We will apply the theory discussed in the previous section to the characteristic function of a L\'{e}vy process. We define a L\'{e}vy process $X$ as in \cite[I. Definition 1]{bertoin1996levy}. The function $\psi:\mathbb{R}\longrightarrow \mathbb{C}$ in $\phi(\lambda)=
 \mathbb{E}[\exp(i \lambda X_t)]=\exp(-t\psi(\lambda))$ where $t\geq0$, is called the characteristic exponent of the L\'{e}vy process $X$ and $\phi$ is the characteristic function which has numerous properties such as uniform continuity and positive definiteness. Let $\tau=\tau(q)$ be an exponentially distributed random time with parameter $q>0$ which is independent of the L\'{e}vy process $X$. The characteristic function of $X_{\tau}$ is $\frac{q}{q+\psi(\lambda)}$, see \cite[p.165]{bertoin1996levy}.

\begin{proposition}\label{posdef}
If $\phi$, $\psi$, $\mu$, and $\sigma$ are the characteristic function, characteristic exponent, drift, and volatility of a L\'{e}vy process respectively and $\alpha\in(0,1)$, then the following hold:

\begin{enumerate}[i)]
\item If $t\neq0$ and $\sigma\neq 0$ then $\phi\in L^2(\R)\cap L^1(\R)+\C 1$ and if $\displaystyle\lim_{\lambda\rightarrow \pm\infty}\psi(\lambda)=\displaystyle\lim_{\lambda\rightarrow \pm\infty}\lambda^2$ then $\frac{q}{q+\psi}\in L^2(\R)\cap L^1(\R)$,
\item $|\phi|$ and $|\frac{q}{q+\psi}+\alpha|$ are bounded and bounded away from zero,
\item $(\frac{q}{q+\psi})^\beta$ is positive definite for all $\beta\in\mathbb{R}^+$,
\item $\Index(\phi+\alpha+1)=\Index(\frac{q}{q+\psi}+\alpha)=\Index(\psi+\alpha)=0$.

\end{enumerate}
\end{proposition}

\begin{proof}
$i)$ By the L\'{e}vy-Khintchine formula,
\begin{equation*}
\psi(\lambda)=\displaystyle\int_{-\infty}^{\infty}(1-\cos(x\lambda))\nu(dx)+\frac{1}{2}\sigma^2\lambda^2+i(\displaystyle\int_{-1}^{1}x\lambda\nu(dx)-\displaystyle\int_{-\infty}^{\infty}\sin(x\lambda)\nu(dx)-\mu\lambda)
\end{equation*}
where $\nu$ is a L\'{e}vy measure. Since $t(\cos(x\lambda)-1)\leq 0$ we conclude $\Re(-t\psi(\lambda))\leq -\frac{1}{2}t\sigma^2\lambda^2$. Therefore,

\begin{equation*}
\displaystyle\int_{-\infty}^{\infty}|\phi(\lambda)|d\lambda\leq\displaystyle\int_{-\infty}^{\infty}\exp(-\frac{1}{2}t\sigma^2\lambda^2)d\lambda
\end{equation*}

\begin{equation*}
\displaystyle\int_{-\infty}^{\infty}\exp(-\frac{1}{2}t\sigma^2\lambda^2)d\lambda=\displaystyle\lim_{\lambda\rightarrow \infty}\frac{\sqrt{2\pi}}{2\sigma\sqrt{t}}\erf(\sigma \sqrt{\frac{t}{2}} \lambda)-\displaystyle\lim_{\lambda\rightarrow -\infty}\frac{\sqrt{2\pi}}{2\sigma\sqrt{t}}\erf(\sigma \sqrt{\frac{t}{2}} \lambda)=\frac{\sqrt{2\pi}}{\sigma\sqrt{t}}
\end{equation*}

Similarly,
\begin{equation*}
\displaystyle\int_{-\infty}^{\infty}|\phi(\lambda)|^2d\lambda=\displaystyle\int_{-\infty}^{\infty}\exp(-2t\Re(\psi(\lambda)))d\lambda\leq\frac{\sqrt{\pi}}{\sigma\sqrt{t}}
\end{equation*}

It is also well-known that a uniformly continuous absolutely integrable function vanishes at infinity. If \cite[Proposition 2 (i)]{bertoin1996levy} holds then $\psi$ behaves like $\lambda^2$ at infinity implying that $\frac{q}{q+\psi}$ goes to zero at infinity like $\frac{1}{\lambda^2}$ and consequently belongs to $L^2(\R)\cap L^1(\R)$.\\
$ii)$ Since $0<|\phi|\leq1$, $\phi$ is bounded and bounded away from zero. Since $\Re(\psi)\geq 0$, we conclude from \cite[Proposition 2]{bertoin1996levy} that $\frac{q}{q+\psi}+\alpha$ is bounded and bounded away from zero. \\
$iii)$ Since $\frac{q}{q+\psi}$ is a characteristic function, it is positive definite by definition. Let $\beta>0$. From the definition of the gamma function we have

\begin{equation*}
(\frac{q}{q+\psi(\lambda)})^{\beta}=\displaystyle\int_0^{\infty}\frac{x^{\beta-1}}{\Gamma(\beta)}\exp(-(\frac{q+\psi(\lambda)}{q})x)dx=\displaystyle\int_0^{\infty}\frac{x^{\beta-1}\exp(-x)}{\Gamma(\beta)}\exp(-(\frac{x}{q})\psi(\lambda))dx
\end{equation*}
The function $\exp(-(\frac{x}{q})\psi(\lambda))$ is positive definite with respect to $\lambda$ and $\frac{x^{\beta-1}\exp(-x)}{\Gamma(\beta)}$ is positive. Therefore, $(\frac{q}{q+\psi(\lambda)})^{\beta}$ is positive definite for all $\beta>0$. \\ 
$iv)$ It is well known that the modulus of the characteristic function is less than or equal to one and since the modulus is also equal to $\exp(-t\Re(\psi)) $ and $t>0$, we conclude that $\Re(\psi)\geq 0$. This fact together with \cite[Proposition 2]{bertoin1996levy} indicates that the functions $\phi+\alpha+1$, $\psi+\alpha$, and $\frac{q}{q+\psi}+\alpha$ are all non-vanishing. Therefore, $\Re(\psi+\alpha)\geq \alpha$ implying that index of $\psi+\alpha$ is zero. Similarly, the curve traced out by $\phi+\alpha+1=\exp(-t\Re(\psi))\exp(-it\Im(\psi))+\alpha+1$ resides in the right half plane and doesn't enclose the origin. We furthermore observe that

\begin{equation*}
\Re(\frac{q}{q+\psi}+\alpha)=(\frac{q}{|q+\psi|})^2+\alpha+\frac{q}{|q+\psi|^2}\Re(\psi)
\end{equation*}

 Since $\Re(\frac{q}{q+\psi}+\alpha)\geq (\frac{q}{|q+\psi|})^2+\alpha$, we conclude that $\Index(\frac{q}{q+\psi}+\alpha)=0$.
\end{proof}


Part iv of the above proposition is just included for its intrinsic interest: we do not actually use the index of $\Psi +\alpha$ for anything.
We next solve the multiplicative Riemann-Hilbert problem for the function  $\frac{q}{q+\psi}$ of the above proposition. This function goes to zero at infinity (quadratically) and therefore does not have an index. Thus, we cannot argue that there exists a unique and positive definite solution(s) to the Riemann-Hilbert problem, but we can use our techniques to at least find a positive definite solution.
The supremum and infimum of the L\'{e}vy process $X_{\tau}$ where $\tau=\tau(q)$ is exponentially distributed are defined to be $M_q=\sup\{X_s: s\leq \tau(q)\}$ and $I_q=\inf\{X_s: s\leq \tau(q)\}$ respectively. It is shown in \cite[VI.2.Theorem 5]{bertoin1996levy} how to solve for the supremum and infimum in terms of a Riemann-Hilbert factorization problem, and evidently the factors must be positive definite to be valid solutions. The next theorem thus gives the distributions of the extrema. This theorem is more general and it is not limited to particular classes of L\'{e}vy processes such as $\alpha$-stable processes considered in \cite{hackkuz2013}.

\begin{theorem}
The distributions $f_{M_q}$ and $f_{I_q}$ of the extrema $M_q$ and $I_q$ of the L\'{e}vy process $X_{\tau}$ where $\tau=\tau(q)$ is exponentially distributed are given as follows:

\begin{enumerate}[i)]
\item $f_{M_q}(\lambda)=\hat{\psi}_{q_+}(\lambda)$ where $\psi_{q_+}(\lambda)=exp((i\lambda+1)g_+(\lambda)-\frac{c}{2})$, and

    \begin{equation*}
    g_+(\lambda)=\lim_{h\longrightarrow0}\sum_{n=-\infty}^{\infty}g\left(nh\right)\frac{\exp\pi i(\lambda/h-n)-1}{2\pi i(\lambda/h-n)}
    \end{equation*}

\item $f_{I_q}(\lambda)=\hat{\psi}_{q_-}(\lambda)$ where $\psi_{q_-}(\lambda)=exp((i\lambda+1)g_-(\lambda)-\frac{c}{2})$, and
    \begin{equation*}
    g_-(\lambda)=\lim_{h\longrightarrow0}\sum_{n=-\infty}^{\infty}g\left(nh\right)\frac{1-\exp-i\pi(\lambda/h-n)}{2\pi i(\lambda/h-n)}
    \end{equation*}

\end{enumerate}
 The function $g$ is defined by $g(\lambda)=\frac{c+\ln(q)-\ln(q+\psi(\lambda))}{1+i\lambda}$.
\end{theorem}
\begin{proof}
 By Proposition \eqref{posdef} $(iii)$, $(\frac{q}{q+\psi(\lambda)})^{\beta}$ is positive definite for all $\beta>0$ implying that $\ln(q)-\ln(q+\psi(\lambda))$ differs by a constant from a positive definite function. Therefore, there exists a constant $c$ such that $c+\ln(q)-\ln(q+\psi(\lambda))$ is positive definite. Although this function is not in $L^2(\mathbb{R})$, we will multiply it by another positive definite function so that the product function belongs to $L^2(\mathbb{R})$. A suitable function is $\frac{1}{1+i\lambda}$ which is positive definite because its inverse Fourier transform is non-negative. According to Theorem \eqref{th:additive.RH}, there exist positive definite functions $g_+$ and $g_-$ such that $g(\lambda)=\frac{c+\ln(q)-\ln(q+\psi(\lambda))}{1+i\lambda}=g_+(\lambda)+g_-(\lambda)$. By exponentiating we have:

\begin{equation*}
\frac{q}{q+\psi(\lambda)}=\exp((i\lambda+1)g_+(\lambda)-\frac{c}{2})\exp((i\lambda+1)g_-(\lambda)-\frac{c}{2})
\end{equation*}

Since multiplication by $(i\lambda+1)$ corresponds to applying a differential operator to the Fourier transform, and since differential operators do not increase support, it follows that $(i\lambda+1)g_+(\lambda)$ and $(i\lambda+1)g_-(\lambda)$ are positive definite. Thus we obtain a factorization $\psi_{q_+}(\lambda) \psi_{q_-}(\lambda)$ of $\frac{q}{q+\psi(\lambda)}.$ It is shown in \cite[VI.2.Theorem 5]{bertoin1996levy} that the first factor is the characteristic function of $M_q$ and the second factor is the characteristic function of $I_q$. The distributions are given by Fourier transform of the characteristic functions. 

\end{proof}

\begin{theorem}
Let $\tau=\tau(q)$ be a geometrically distributed random time with parameter $0<q<1$ and let $\phi(\lambda)=\exp(-\psi(\lambda))$ be the characteristic function of the L\'{e}vy process $X_1$ such that $\displaystyle\lim_{\lambda\rightarrow \pm\infty}\psi(\lambda)=\displaystyle\lim_{\lambda\rightarrow \pm\infty}\lambda^2$. Then, the characteristic function of $X_{\tau}$ which is of the form  $f(\lambda)=\frac{1-q}{1-q\phi(\lambda)}$ has the following properties:
\begin{enumerate}[i)]
	\item $f\in L^2(\R)+\C 1,$ \item $|f|$ is bounded and bounded away from zero, \item $f^\beta$ is positive definite for all $\beta>0$.
\end{enumerate}
\end{theorem}

\begin{proof}
$i)$ The function $f$ decays to $1-q$ like an exponential. Therefore, $f\in L^2(\R)+\C 1$.

$ii)$ Since $0<|\phi(\lambda)|\leq 1$, it follows that $1-q<|f(\lambda)|\leq1$ for all $\lambda\in \mathbb{R}$.

$iii)$ By the binomial series, we have

\begin{equation*}
\frac{1}{(1-q\phi(\lambda))^\beta}=\sum_{k=0}^{\infty}q^k \dbinom{k+\beta-1}{k}(\phi(\lambda))^k
\end{equation*}

 Since $\dbinom{k+\beta-1}{k}=\frac{\Gamma(k+\beta)}{\Gamma(k+1)\Gamma(\beta)}=\frac{\Gamma(k+\beta)}{k!\Gamma(\beta)}$ is positive, $q^k$ and $(1-q)^\beta$ are positive for all $\beta>0$ and $(\phi(\lambda))^k=\exp(-k\psi(\lambda))$ is positive definite,  we conclude that $f^\beta$ is positive definite for all $\beta>0$.
\end{proof}

The following corollary follows from the above Theorem together with Theorem \eqref{main}:

\begin{corollary}
Let $\tau=\tau(q)$ be a geometrically distributed random time with parameter $0<q<1$ and let $\phi(\lambda)=\exp(-\psi(\lambda))$ be the characteristic function of the L\'{e}vy process $X_1$ such that $\displaystyle\lim_{\lambda\rightarrow \pm\infty}\psi(\lambda)=\displaystyle\lim_{\lambda\rightarrow \pm\infty}\lambda^2$. Then, the characteristic function of $X_{\tau}$, i.e. $f(\lambda)=\frac{1-q}{1-q\phi(\lambda)}$ has a factorization of the form $f(\lambda)=\Phi_{+}\Phi_-$ where $\Phi_+$ is positive definite on $\R$ and extends to a bounded holomorphic function on the upper half plane, and $\Phi_-$ is positive definite on $\R$ and extends to a bounded holomorphic function on the lower half plane.
\end{corollary}

\begin{proposition}
Let $(L^{-1},H)$ be the ladder process of a given L\'{e}vy process $X$ with characteristic function $\phi(\alpha,\beta)=\mathbb{E}[\exp(i\alpha L^{-1}(t)+i\beta H(t))]=\exp(-\kappa(\alpha,\beta)t)$ where $t\geq 0$ and the bivariate Laplace exponent $\kappa(\alpha,\beta)$ is of the following form \cite[VI.2.Corollary 10]{bertoin1996levy}:

\begin{equation*}
\kappa(\alpha,\beta)=k\exp\bigg(\displaystyle\int_{0}^{\infty}dt\displaystyle\int_{[0,\infty)}(e^{-t}-e^{-\alpha t-\beta x})t^{-1}\mathbb{P}(X_t\in dx)\bigg)
\end{equation*}
Then, the following hold:
\begin{enumerate}[i)]
\item $|\phi|$ is bounded and bounded away from zero,
 \item $\phi^\gamma$ is positive definite for all $\gamma>0$.
\end{enumerate}
\end{proposition}
\begin{proof}
$i)$ Since $0<|\phi|\leq1$, $\phi$ is bounded and bounded away from zero.\\

$ii)$ Since $\phi$ is a characteristic function, it is positive definite. Hence, $\kappa(\alpha,\beta)$ is negative definite. Since $\gamma>0$ , $\gamma\kappa(\alpha,\beta)$ is negative definite which implies that $\phi^{\gamma}$ is positive definite.



\end{proof}
\bibliographystyle{plain}

\end{document}